\documentclass[12pt,a4paper]{amsart}
\usepackage[utf8]{inputenc}
\usepackage[style=numeric,backref]{biblatex}%
\renewbibmacro{in:}{\ifentrytype{incollection}{%
		\printtext{\bibstring{in}\intitlepunct}}{%
		\ifentrytype{inproceedings}{%
			\printtext{\bibstring{in}\intitlepunct}}{}%
}}
\renewbibmacro*{volume+number+eid}{\printfield{volume}%
	\setunit*{\addnbthinspace}\printfield{number}%
	\setunit{\addcomma\space}\printfield{eid}%
}
\DeclareFieldFormat[article]{number}{\mkbibparens{#1}}
\DefineBibliographyStrings{english}{%
	bibliography = {References},
}
\usepackage{pythonhighlight}
\usepackage{etoolbox}
\newtoggle{omegas}
\togglefalse{omegas}

\usepackage{amsmath,amsfonts,amssymb,amsthm,graphicx,dsfont}
\usepackage{bbm,xcolor,enumitem, lipsum}
\usepackage[a4paper]{geometry}
\DeclareSymbolFont{bbold}{U}{bbold}{m}{n}
\DeclareSymbolFontAlphabet{\mathbbold}{bbold}

\newtheorem*{theorem*}{Theorem}

\newtheorem{theorem}{Theorem}[section]
\newtheorem{claim}[theorem]{Claim}
\newtheorem{lemma}[theorem]{Lemma}
\newtheorem{corollary}[theorem]{Corollary}
\theoremstyle{definition}
\newtheorem*{definition}{Definition}

\newtheorem{remark}[theorem]{Remark}

\newtheorem*{remark*}{Remark}
\newtheorem*{lemma*}{Lemma}
\newtheorem*{corollary*}{Corollary}

\newtheorem*{namedthm}{\protect\namedthmname}
\newcounter{namedthm}

\makeatletter
\newenvironment{named}[1]
{\def\namedthmname{#1}%
	\refstepcounter{namedthm}%
	\namedthm\def\@currentlabel{#1}}
{\endnamedthm}
\makeatother

\newcommand{\PR}{{\mathcal P}}
\newcommand{\FF}{{\mathbb F}}

\newcommand{\cf}{c_{\overline{\mu}}}
\newcommand{\cs}{c_{*}}

\DeclareMathOperator{\E}{\mathbb E}
\renewcommand{\P}{\operatorname{\mathbb P}}

\usepackage{parskip,mathrsfs,xtab}

\usepackage{setspace}
\usepackage[nodayofweek]{datetime}

\newdateformat{monthyear}{\monthname[\THEMONTH], \THEYEAR}
\usepackage{hyperref}
\hypersetup{
	pdftitle={Squarefree polynomials over function fields},
	pdfsubject={math.NT},
	pdfauthor={Naomi Bazlov},
	pdfkeywords={Analytic number theory, Squarefree polynomials},
}
\title{Squarefree polynomials with missing digits}
\author{Naomi Bazlov}
\email{naomi.bazlov@campus.technion.ac.il}
\begin{document}
	\maketitle
	\begin{abstract}
		We establish an asymptotic formula for squarefree polynomials over a finite field whose coefficients are in $\{0,1\}$. This resolves a function-field analogue of a conjecture of Erd\H{o}s, Mauduit and S\'ark\"ozy. More generally, we estimate the probability of a polynomial $\sum \varepsilon_it^i$ being squarefree, where $\varepsilon_i$ are random variables sampled according to measures on $\FF_q$.
	\end{abstract}
		\section{Introduction}
		In number theory, it is a classical result that the natural density of squarefree integers is $1/\zeta(2)=6/\pi^2\approx 0.6079$. Function fields exhibit the same phenomenon: the proportion of polynomials in $\FF_q[t]$ that are squarefree tends to $1-1/q$ as the degree grows, which is the reciprocal of $\zeta_q(2)$, the function field analogue. 
		The connection between 
		squarefree integers and squarefree polynomials has been well studied: see, for example,~\cite{Poonen_2003},~\cite{oppshu} and~\cite{CARMON}.
		
We consider how the squarefree property is affected when we impose additional constraints on polynomial coefficients. Specifically, we restrict each coefficient to its own set of allowable values, 
		and ask how many polynomials under these restrictions remain squarefree. This question is motivated by analogous problems 
		for integers with certain digits forbidden in their base-$q$ expansion.

		The following conjecture, made in 1998 by Erd\H{o}s, Mauduit, S\'ark\"ozy~\cite{Erdos1998},
is an important motivating question for our paper -- proved for $g\le 5$ in 1996~\cite{Filaseta1996SquarefreeVO} by Filaseta and Konyagin.
		\begin{named}{Conjecture}\label{conj}
			Given a base $g\in\mathbb{N}_{\ge 2}$, there are infinitely many
			squarefree integers such that every digit of them in the number system to
			base $g$ is 0 or 1.
		\end{named}
		No further progress has been made since the results in~\cite{Filaseta1996SquarefreeVO}. 
				Erd\H{o}s, Mauduit, S\'ark\"ozy~\cite{Erdos1998} were able to ban up to $g-\sqrt{g}$ digits and obtain results on the distribution of squarefree integers. Our main result 
				 can be considered a resolution of the conjecture in the setting of polynomials over finite fields:
\begin{theorem}\label{thm: main result} Let $q=p^k$ be fixed, $n\to\infty$.
	The number of squarefree, monic polynomials in $\mathbb{F}_q[t]$ of degree $n$ with coefficients only $0$ or $1$ is
	\[\frac34\cdot\frac1{1+1/p}\cdot 2^n(1+O_q(p^{-n^{1/4}})).\]
\end{theorem}
	To generalise this theorem, we ask ourselves the following:
		\begin{named}{General Question}\label{question: gen}
			Let $q$ be a fixed prime power, and $n \in \mathbb{N}$ a variable tending to infinity. If a random monic polynomial $f \in \mathbb{F}_q[t]$ of degree $n$ is generated by sampling each coefficient independently according to some measures on $\FF_q$, what is the asymptotic behavior of $\mathbb{P}(f \text{ is squarefree})$?
		\end{named}	
In the case of a uniform measure over $\FF_q$, the probability of a polynomial of degree $\ge 2$ being squarefree is $1/\zeta_q(2)$, where $\zeta_q(s)=	\sum_{a\in\mathbb{F}_q[t]\text{ monic}}(q^{\deg a})^{-s}=\sum_{k\ge 0}q^{k-ks}=\frac1{1-q^{1-s}}$. See a short proof of this in~\cite[Remark 11.2.17]{Handbook}. However, we would like to consider probability measures with small support; in terms of coefficients, we would like to ban as many as possible (assigning them probability $0$).
		In Theorem~\ref{thm: main result} we banned all but two coefficients, and in fact, our method works for any two allowed coefficients in $\FF_p$. However, when studying polynomials with more generic allowed coefficients inside $\FF_q$, our result requires additional conditions on the set. These conditions are automatically satisfied when $q$ is a prime.
		
		The following theorem can be considered an answer to the~\ref{question: gen}:
		\begin{theorem}\label{thm: general measure squarefree}
			Let $n\to\infty$, and $\mu_0, \mu_1, \dots$ be a sequence of probability measures on $\mathbb{F}_q$, with $\overline{\mu}=(\mu_0,\dots,\mu_{n-1})$. Suppose there exists a constant $\cs<1$ such that the maximum non-trivial Fourier amplitude
			 satisfies $$\cf:=\max_{0\le i\le n-1} \max_{\alpha \in \mathbb{F}_q^\times} |\hat{\mu}_i(\alpha)|\le	\sup_{i\ge 0} \max_{\alpha \in \mathbb{F}_q^\times} |\hat{\mu}_i(\alpha)| \le \cs.$$ (Here $\hat{\mu_i}$ denotes the discrete Fourier transform, so 
			 $\hat{\mu_i}(\alpha) = \sum \mu_i(\xi) \exp\bigl(\frac{2\pi i\mathrm{tr}(-\alpha\xi)}p\bigr)$.)
			
			Then the probability that a random monic polynomial $f(t) = t^n + \sum_{i=0}^{n-1} \varepsilon_i t^i$, with each coefficient $\varepsilon_i$ drawn independently according to $\mu_i$, is squarefree is given by:
			\[ \mathbb{P}_{\overline{\mu}}(f \text{ is squarefree}) = \kappa_{\overline{\mu}} + O_{q,\cs}(q^{-n^{1/4}}), \]
			where the asymptotic density $\kappa_{\overline{\mu}}$ is:
			\[ \kappa_{\overline{\mu}} = \frac{1}{1+1/q} \big(1 - \mu_0(0)\mu_1(0)\big). \]
		\end{theorem}
		\begin{remark}
	The constant $c_0=\frac{1}{1+1/q}=\zeta_q^{-1}(2) (1-1/q^2)^{-1}$ represents the the Euler product of the local squarefree factors over all monic 	irreducible $P\ne t$ (see~\eqref{c_0}), and $\Lambda=\mu_0(0)\mu_1(0)$ is the probability that $t^2 \mid f$ (see~\eqref{prob.tsquared}). Since the first two coefficients of $f\in\mathbb{F}_q[t]$ being $0$ is equivalent to $t^2\mid f$, this strongly affects whether the polynomial is squarefree. 
\end{remark}
When $\cf=1$, in many cases this completely changes the probabilities. For example, if $\overline{\mu} = (\mu_0, \dots, \mu_{n-1})$ is defined so that $\mu_i(0)=1$ for $p\nmid i$ (the Dirac delta measure) and $\mu_{kp}$ are uniform on all of $\FF_q$, we get $\mathbb{P}_{\overline{\mu}}(f \text{ is squarefree})=0$ when $n$ is a multiple of $p$. We will talk more about $\cf=1$ in Section~\ref{section: c=1}.
	
From Theorem~\ref{thm: general measure squarefree} we recover counting formulae for squarefree polynomials with restricted coefficients by using probability measures uniform on a subset of $\FF_q$.
\begin{corollary}\label{thm: gen sec 4}
	Let $\mathbb{F}_{p^k}\subset \FF_q$ be a subfield, and let $\overline{\mathcal{R}}=(\mathcal{R}_0,\dots,\mathcal{R}_{n-1})$, with each $\mathcal{R}_{i}\subset \mathbb{F}_{p^k}$, be an ordered collection of subsets of size $s$ such that $\mathbb{F}_{p^k}\setminus\mathcal{R}_i$ is not contained in a proper affine subspace of $\FF_{p^k}$. The number of squarefree, monic polynomials of degree $n$ with the coefficient of $t^{i}$ only from $\mathbb{F}_{p^k} \backslash \mathcal{R}_i$ is given by
	\[\kappa_{\overline{\mathcal{R}}}(p^k-s)^n(1+O_{p^k,\cs}(p^{-kn^{1/4}})),\] where
	$\kappa_{\overline{\mathcal{R}}}=\begin{cases}
		\frac1{1+1/p^k}&\text{if } 0 \in \mathcal{R}_0 \text{ or } 0 \in \mathcal{R}_1\\
		\frac1{1+1/p^k}\cdot (1-\frac1{(p^k-s)^2})&\text{if } 0 \not\in \mathcal{R}_0 \text{ and } 0 \not\in \mathcal{R}_1.
	\end{cases}$ 
\end{corollary}		
In Section~\ref{section: restrict coeff intro}, we prove 
exponential sum bounds related to counting polynomials with missing digits divisible by a given polynomial. The results are inspired by a paper of Porritt~\cite{Porritt2019}, who studied prime polynomials over $\FF_q$ whose non-leading coefficients are drawn from a subset of $\mathbb{F}_q$ missing $s$ elements (our measure theory approach generalises this). Porritt adapted some methods from Maynard~\cite{Maynard2021} who  made great progress on the problem of integer primes with restricted digits. In the integers, current methods allow at most $O(q^{2/5})$ forbidden digits in large bases $q$~\cite{granville2023}, but in $\FF_q[t]$, Porritt can ban up to $\sqrt{q}/2$ digits. His work was also extended by Cheng~\cite{CHENG2025} to show that instead of all coefficients being drawn from the same subset, they can be drawn from different restricted subsets of $\mathbb{F}_q$, each missing $s$ `digits'.

In Section~\ref{section: restrict coeff main} we apply some methods of dealing with polynomials over $\mathbb{F}_q$ used in~\cite{CARMON}, 
and we prove Theorem~\ref{thm: general measure squarefree} by combining the Brun sieve for small prime factors with a probabilistic 
bound for large prime factors, as in~\cite{universality}. We derive an explicit expression for the proportion of squarefree polynomials, up to an $O_{q,\cs}(q^{-n^{1/4}})$ error term. Section~\ref{section: ban more} addresses the case where the coefficients' random choice is uniform, giving in Corollary~\ref{lem: affine subfield} an asymptotic formula for the number of squarefree polynomials when each coefficient set is not contained in a proper affine subspace of $\FF_q$. We discuss special cases where the choice of coefficients is restricted to small sets with a specific additive structure imposed; most importantly, when the only allowed coefficients are 0 or 1, proving Theorem~\ref{thm: main result}. Sections~\ref{sec: vary all} and~\ref{section: large char} discuss what happens when $q$ is not fixed, and finally, Section~\ref{section: c=1} analyses some special cases when $\cf=1$.
		\section{Setup of the Problem}\label{section: restrict coeff intro}
		In this section, we explain the setup needed to tackle the~\ref{question: gen}.
		As in~\cite{Porritt2019}, we work inside the field of Laurent series $\FF_q((1/t))$, which is the $1/t$-adic completion of $\mathbb{F}_q(1/t)$. The subring $\FF_q[[1/t]]$ has discrete valuation arising from the norm $\|f/g\|=q^{\deg f-\deg g}$ for nonzero $f$, $0$ for $f=0$, for $f/g$ a rational function, and the unique maximal ideal is
		$$
		\mathbb{T}=\frac1t\mathbb{F}_q[[1/t]].
		$$
The following straightforward fact will be useful in some of our proofs; it is implicit in many works on the circle method in $\FF_q[t]$, for instance in Porritt~\cite[Lemma 3]{Porritt2019}, who was inspired by Pollack~\cite[Lemma 7]{Pol}.
\begin{claim}\label{claim: d zeroes}
	If $a,g\in\FF_q[t]$ are coprime, $\deg a<\deg g=d$, $g$ is monic and not a power of $t$, then $a/g$ has no $d$ consecutive zeroes in its Laurent series expansion.
\end{claim}
\begin{proof}
Let $a/g=\sum_{i<0}\gamma_it^i$. Suppose $\gamma_{-c-1},\dots,\gamma_{-c-d}=0$ for some $c\ge 0$. Then $\|\{t^ca/g\}\|\le q^{-d-1}$, where $\{\cdot\}$ denotes the fractional part of the Laurent series -- with negative powers of $t$. But this is impossible as $g$ does not divide $t^ca$, so dividing with remainder gives a nonzero result with norm at least $q^{-d}=\|1/g\|$.
\end{proof}
We make use of the following throughout this paper:
\begin{definition}\
\begin{itemize}
\item Define $\psi \colon \mathbb{F}_q \rightarrow \mathbb{C}^\times$ by $\psi(a) = \exp(2\pi i\mathrm{tr}(a)/p)$, where $\text{tr} \colon \mathbb{F}_q \rightarrow \mathbb{F}_p$ is the usual trace map. 
\item Define also the additive character $\textbf{e}_q \colon \mathbb{F}_q((1/t))\rightarrow \mathbb{C}^\times$ by $\textbf{e}_q (x) = \psi(x_{-1})$.
\end{itemize}\end{definition}
		Pick, for each $i$, a probability measure $\mu_i$ on $\FF_q$ for the coefficient of $t^i$ in the polynomial, and then independently choose all  coefficients according to the respective distributions.
We now introduce some relevant notation:
\begin{definition}
	Let 
	$\overline{\mu}=(\mu_0,\mu_1,\dots,\mu_{n-1})$ be a tuple of probability measures on $\FF_q$. Consider $\overline{\mu}$ as a (finite) measure on $\mathcal{M}_n$, the (finite) set of monic polynomials in $\FF_q[t]$ of degree $n$, via $\overline{\mu}(f)=\prod_i\mu_i(f_i)$. Here 
	$$f=t^n+\sum_if_it^i\in\mathcal{M}_n,$$ where $f_i$ are sampled according to $\mu_i$.
	
	For $x\in\mathbb{T}$, define the Fourier transform of our random $f$ as
	\[S_{\overline{\mu}}(x)=\mathbb{E}_{f\sim \overline{\mu}}\mathbf{e}_q(fx)=\sum_{f\in\mathcal{M}_n}\overline{\mu}(f)\textbf{e}_q(fx).\]
\end{definition}
This definition appeared in a previous work on universality questions~\cite{universality}, and this expectation was also used in~\cite[Part II]{partII}.
		
		We adapt~\cite[Lemma 3]{Porritt2019} to fit our setting:
\begin{lemma}\label{lemma: measure theory}
	Let $\mu_i$ be probability measures on $\mathbb{F}_q$, $\overline{\mu}=(\mu_0,\dots,\mu_{n-1})$.	Let $a,g\in\mathbb{F}_q[t]$ be coprime polynomials, with $\|a\|<\|g\|$, $\deg g=d$, $g$ monic and not a power of $t$. Write $a/g=\sum_{i<0}\alpha_it^i$.
	Then
	\[|S_{\overline{\mu}}(a/g)|\le  \cf^{\lfloor \frac nd \rfloor},\] where
	$\cf=\max_i\max_{\alpha\in\mathbb{F}_q^\times}|\hat{\mu_i}(\alpha)|$,
	where $\hat{\mu_i}$ is the discrete Fourier transform of the measure $\mu_i$, such that
	$\hat{\mu_i}(\alpha) = \sum \mu_i(\xi) \psi(-\alpha\xi)$.
\end{lemma}
\begin{proof}
			Since $\|a\|<\|g\|$, we may indeed write $a/g=\sum_{i<0}\alpha_it^i$; dividing both numerator and denominator by $t^{\deg g}$ shows all $\alpha_i$ for $i\ge 0$ are $0$. (Alternatively, one may look at the norm.) So in particular, $a/g$ is an element of $\mathbb{T}$.
			By definition of $S_{\overline{\mu}}$, we have
			$S_{\overline{\mu}}(a/g) = \mathbb{E}_{f \sim \overline{\mu}} [\mathbf{e}_q(fa/g)]$. Note that the leading term of any sampled $f$ is $t^n$, as these are monic polynomials of degree $n$. Hence, by additivity of the character $\textbf{e}_q$, we have for $x=a/g\in\mathbb{T}$ that $\textbf{e}_q(fx)=\textbf{e}_q(t^n\cdot x)\textbf{e}_q(f_{n-1}t^{n-1}\cdot x)\textbf{e}_q(f_{n-2}t^{n-2}\cdot x)\dots$. Since the coefficients $f_i$ are sampled independently according to the measures $\mu_i$, the expectation of the product is the product of the expectations:	
			$$S_{\overline{\mu}}(a/g) = \mathbf{e}_q(t^n \cdot a/g) \prod_{i=0}^{n-1} \mathbb{E}_{f_i \sim \mu_i} [\psi(f_i \alpha_{-i-1})]$$
			
			Here we used that $\textbf{e}_q(x)=\psi(x_{-1})$, where $\psi(a)=\exp(2\pi i\mathrm{tr}(a)/p)$. Note that by definition of $\textbf{e}_q$, it has absolute value $1$, so 
			\[ |S_{\overline{\mu}}(a/g)|=\prod_{i=0}^{n-1}\left|\mathbb{E}_{f_i \sim \mu_i} [\psi(f_i \alpha_{-i-1})]\right|.\]
Each factor in the product is the Fourier transform of $\mu_i$ evaluated at the Laurent coefficient $\alpha_{-i-1}$, which is $\hat{\mu}_i(-\alpha_{-i-1})$. If $\alpha_{-i-1} = 0$, then $\hat{\mu}_i(0) = \sum_{y \in \mathbb{F}_q} \mu_i(y)= 1$. If $\alpha_{-i-1} \neq 0$, then by definition $|\hat{\mu}_i(\alpha_{-i-1})| \leq \cf = \max_i \max_{x \neq 0} |\hat{\mu}_i(x)|$.
			
Let us say that $\alpha_{-i-1}\ne 0$ for $z$ different values of $i$, giving a factor of $\cf^z$ for these $i$, and the coefficients that are $0$ give a factor of $1$. Now, if we split the coefficients $\alpha_{-n},\dots,\alpha_{-1}$ into $\lfloor \frac nd \rfloor$ groups of consecutive coefficients, each has nonzero coefficients by Claim~\ref{claim: d zeroes}, so $z\ge \lfloor \frac nd \rfloor$.
		\end{proof}	 Note that taking measures $\mu_i$ as uniform probability measures $\rho_{i,\mathcal{R}_i}$ with supports $\FF_q\setminus\mathcal{R}_i$ of size $q-s$, we get back the kind of setting studied by Porritt.
		 See also~\cite[Part II]{partII}, where a similar setup is used to bound Fourier transforms of measures.
\section{Asymptotics of squarefrees under general measures}\label{section: restrict coeff main}
		
We denote by $\mathcal{P}$ the set of monic irreducible polynomials in $\FF_q[t]$, and by $\mathcal{P}^{<k}$, $\mathcal{P}^{\le k}$, $\mathcal{P}^{\ge k}$, the set of irreducible polynomials of degree $<k$, $\le k$ and $\ge k$ respectively.
		
		We wish to estimate the probability that a random monic polynomial $f$ of degree $n$, with coefficients sampled according to the product probability measure $\overline{\mu} = \prod_{i=0}^{n-1} \mu_i$, is squarefree.
		
		Let 
		$N=\{f\text{ is squarefree}\} = \{f \in \mathcal{M}_n : \forall P \in \mathcal{P}, P^2 \nmid f\}.$
		To estimate $\mathbb{P}_{\overline{\mu}}(N) = \mathbb{E}_{\overline{\mu}}[\mathbbm{1}\{\forall P \in \mathcal{P}: P^2 \nmid f\}]$, we introduce two auxiliary events:
		\begin{itemize}
			\item $N' = \bigcap_{P \in \mathcal{P}^{<m_1}} \{f \in \mathcal{M}_n : P^2 \nmid f\}$ is the event that $f$ is not divisible by the square of any ``small'' prime.
			\item $N'' = \bigcup_{P \in \mathcal{P}^{\ge m_2}} \{f \in \mathcal{M}_n : P^2 \mid f\}$ is the event that $f$ is divisible by the square of at least one ``large'' prime.
		\end{itemize}
%
Any squarefree $f$ is, in particular, not divisible by the square of any prime of bounded degree, so $N \subseteq N'$. Setting $m_1 \ge m_2$, we have $N \supseteq N' \setminus N''$, as polynomials not divisible by squares of any `small' prime must either be squarefree or be divisible by the square of some `large' prime. Hence
\begin{equation}\label{eq: sec 3 start}\mathbb{P}_{\overline{\mu}}(N') - \mathbb{P}_{\overline{\mu}}(N'') \le \mathbb{P}_{\overline{\mu}}(N) \le \mathbb{P}_{\overline{\mu}}(N').\end{equation}
		
		We will show that $\mathbb{P}_{\overline{\mu}}(N'')$ is negligible. Together with the two inequalities, this tells us that being squarefree is asymptotically indistinguishable from not being divisible by the square of a `small' prime.
		\begin{itemize}
			\item To find the main term from the `small' primes, it is natural to approach using combinatorial techniques such as the inclusion-exclusion principle, but we will make this more efficient by using Brun's pure sieve. We  estimate local probabilities $\mathbb{P}_{\overline{\mu}}(D^2 \mid f)$ by evaluating the exponential sums via Lemma~\ref{lemma: measure theory}.
			\item To handle $N''$ and show that the ``large'' primes give a negligible contribution, we translate the quadratic condition $P^2 \mid f$ into the equivalent linear conditions $P \mid f$ and $P \mid f'$. This is an advantage of working with polynomials rather than integers; it has been successfully used in works on squarefrees, for instance by Poonen~\cite{Poonen_2003} and Carmon~\cite{CARMON}. We proceed as follows:
			\begin{enumerate}[label=(\roman*)]
				\item Let $T_n$ be the set of all possible derivatives of polynomials in $\mathcal{M}_n$. A union is bounded by a sum, thus
				$$\mathbb{P}_{\overline{\mu}}(N'') \le \mathbb{P}_{\overline{\mu}}(f' = 0) + \sum_{d\in T_n} \mathbb{P}_{\overline{\mu}}(f'=d) \sum_{\substack{P\mid d \\ P\in\mathcal{P}^{\ge m_2}}} \mathbb{P}(P \text{ divides } f \mid f'=d)$$
				where we also rearranged to first sum over all possible derivatives $d$. We can do this because any $f$ which contributes to $\mathbb{P}_{\overline{\mu}}(N'')$ has derivative in $T_n$, and $P$ is a factor of $f, f'$ so all of its contribution has been counted.
				\item By fixing the derivative $f'$, we determine all coefficients except those corresponding to $p$-th powers. We then use a probabilistic H{\"o}lder-type bound established in Lemma~\ref{lemma: halasz-bound modified} to show that $\mathbb{P}_{\overline{\mu}}(P\text{ divides } f \mid f'=d)$ decays exponentially with the degree of $P$.
				\item By counting large prime factors of $d$, we bound the number of summands in the inner sum by $\frac{n-1}{m_2}$.
				
				There is an edge case: the derivative vanishes ($d=0$), i.e. $f$ is a perfect $p$-th power, which means coefficients not associated with $p$-th powers are exactly zero; this is dealt with separately. 
			\end{enumerate}
		\end{itemize}
		\subsection{Estimating the asymptotics}	
		
		Note that the probability $\mathbb{P}_{\overline{\mu}}(D^2\mid f)$, where $D$ is squarefree and $f\in\mathcal{M}_n$ is sampled according to $\overline{\mu}$, is given by the expectation of the indicator function, which can be expanded using orthogonality relations of additive characters as below. These were first applied to polynomials over finite fields by Hayes~\cite[Theorem 3.7]{Hayes1966}.
		\begin{equation*}
			\mathbb{P}_{\overline{\mu}}(D^2\mid f) = \mathbb{E}_{\overline{\mu}}[\mathbbm{1}_{D^2\mid f}] = \sum_{f\in\mathcal{M}_n}\overline{\mu}(f)\frac1{\|D^2\|}\sum_{\deg a<\deg D^2}\textbf{e}_q\left( \frac{af}{D^2}\right).
		\end{equation*}
		Swapping the order of summation gives an expression in terms of the exponential sum $S_{\overline{\mu}}$:
		\begin{equation}\label{Sr expression}
			\mathbb{P}_{\overline{\mu}}(D^2\mid f) = \frac1{\|D^2\|}\sum_{\deg a<\deg D^2}S_{\overline{\mu}}\Bigl(\frac{a}{D^2}\Bigr).
		\end{equation}
		
		When $a=0$, we get the main term $\frac1{\|D^2\|}$. 
		For the terms where $a \ne 0$, we take out common factors of $a$ and $D^2$, so that the numerator and denominator are coprime as is necessary to apply Lemma~\ref{lemma: measure theory}.
		However, Lemma~\ref{lemma: measure theory} is not applicable when the denominator is a power of $t$, so we handle divisibility by $t^2$ separately. Since the coefficients $f_0\sim\mu_0$ and $f_1\sim\mu_1$ are chosen independently, 
		we have:
		\begin{equation}
			\mathbb{P}_{\overline{\mu}}(t^2\mid f)=\mathbb{P}_{\overline{\mu}}(f_0=f_1=0)=:\Lambda= \mu_0(0)\mu_1(0). \label{prob.tsquared}
		\end{equation}
		
\begin{lemma}\label{lem: estimating N'}
	Let $N'$ be the event that $f$ is not divisible by the square of any prime $P$ with $\deg P < m_1$. For large enough $r > 0$ such that $r<m_1/2$,
	\[ \mathbb{P}_{\overline{\mu}}(N') = \kappa_{\overline{\mu}}+ O\left(\frac{1}{m_1q^{m_1}}\right) + O((\gamma r)^{-r}) + O\left( \left(\frac{q^{m_1}}{m_1}\right)^r 3^r \cf^{\frac{n}{2rm_1}} \right), \]
	where $\kappa_{\overline{\mu}} = \frac{1}{1+1/q}(1-\mu_0(0)\mu_1(0))=\frac{1}{1+1/q}(1-\Lambda)$, and $\gamma$ is a suitable absolute constant.
\end{lemma}
\begin{proof}Let $\mathcal{P}(k)=\prod_{P \in \PR^{< k}}P$. Define $$N_r' = \sum_{j=0}^{r}(-1)^j\sum_{\substack{D\mid\mathcal{P}(m_1)\\ \omega(D)=j}}\mathbb{P}_{\overline{\mu}}(D^2\mid f).$$
	By the inclusion-exclusion principle, $\mathbb{P}_{\overline{\mu}}(N') = N'_{\infty}.$
In fact, by the Pure Brun Sieve (see~\cite[Chapter 6]{CojMur}), $\mathbb{P}_{\overline{\mu}}(N') \le N'_r$ for all even $r$ and $\mathbb{P}_{\overline{\mu}}(N') \ge N'_r$ for all odd $r$. So to find the main term and error of $\mathbb{P}_{\overline{\mu}}(N')$, it suffices to prove them for some sufficiently large (even and odd) $r$ which will upper and lower bound our expression. 
		
		From our exponential sum formulation, for $t\nmid D$, we have:
		\begin{align*}
			\mathbb{P}_{\overline{\mu}}(D^2\mid f) &= \frac{1}{\|D^2\|} + \mathrm{Err}_{\overline{\mu}}(n,D),
		\end{align*} 
		where the error can be bounded by applying Lemma~\ref{lemma: measure theory} to the non-zero frequencies in~\eqref{Sr expression}. For example, if $D=P$ is a prime of degree $k$, the error is $\displaystyle O\Bigl(\cf^{\lfloor \frac n{2k}\rfloor}+\frac{\cf^{\lfloor \frac n{k}\rfloor}}{q^k}\Bigr)$, where $\cf < 1$ is the maximum Fourier amplitude defined in Lemma~\ref{lemma: measure theory}.
		
		Note that for $t\mid D$, the condition $D^2\mid f$ is equivalent to $t^2\mid f$ and $(D/t)^2\mid f$. Because $D$ is squarefree, it contains at most one factor of $t$ (with multiplicity). As in~\eqref{prob.tsquared}, the condition $t^2\mid f$ occurs with probability $\Lambda$. 
		We write $f(t) = t^2 g(t)$, where $g(t) = t^{n-2} + \sum_{i=0}^{n-3} f_{i+2} t^i$. Since $\gcd(t, D/t) = 1$, the condition $(D/t)^2 \mid f(t)$ is equivalent to $(D/t)^2 \mid g(t)$. Because the coefficients of $g(t)$ are sampled independently according to the measures $\mu_2, \dots, \mu_{n-1}$, we can argue as above, except we are working with a degree $n-2$ polynomial.
		Thus
		\[ \mathbb{P}_{\overline{\mu}}(D^2 \mid f) = \Lambda \left( \frac{1}{\|D/t\|^2} + \mathrm{Err}_{\overline{\mu}}(n-2,D/t) \right), \] where the error calculation will follow the same logic as above but replacing $n$ with $n-2$ and $k$ with $k-1$ due to $D/t$ having degree lowered by 1.
		
		Following the method in~\cite{CARMON}, we first focus on the main term of $N_r'$, which is
		$$N_r = \sum_{{\substack{D \mid \mathcal{P}(m_1)\\ \omega(D) \le r\\ t\nmid D}}} \frac{\mu_\mathrm{Mob}(D)}{\|D\|^2}+\sum_{{\substack{D \mid \mathcal{P}(m_1)\\ \omega(D) \le r\\ t\mid D}}} \frac{\mu_\mathrm{Mob}(D)}{\|D/t\|^2}\cdot \Lambda;$$
		here $\mu_\mathrm{Mob}(D)=(-1)^{\omega(D)}$ is the  M\"{o}bius function, and $\Lambda$ is as defined in~\eqref{prob.tsquared}.
		
		If $t\mid D$, let $D=tD_2$, where $D_2$ is coprime to $t$. We have $\omega(D_2)=\omega(D)-1$, and~$\mu_\mathrm{Mob}(D_2)=-\mu_\mathrm{Mob}(D)$.
		Thus if we define $$U(r,m_1) = \sum_{{\substack{D\mid \mathcal{P}(m_1)\\ \omega(D)\le r\\ t\nmid D}}}\frac{\mu_\mathrm{Mob}(D)}{\|D\|^2},$$ our goal is to estimate $U(r,m_1)$ as
		$$N_r = U(r,m_1)-\Lambda U(r-1,m_1).$$
%
		We may further define
		\begin{align}
			U(\infty,m_1) &  = \sum_{{\substack{D\mid \mathcal{P}(m_1)\\ t\nmid D}}}\frac{\mu_\mathrm{Mob}(D)}{\|D\|^2}
			= \prod_{\substack{P \in \PR^{<m_1}\\P\ne t}} \left(1-\frac{1}{\|P\|^2}\right) 
			\notag\\
			& = c_0 \prod_{P \in \PR^{\ge m_1}} \left(1-\frac{1}{\|P\|^2}\right)^{-1}
			= c_0\left(1+O\Bigl(\sum_{P\in\mathcal{P}^{\ge m_1}}\frac1{\|P\|^2}\Bigr)\right), \label{eq: Uinfty}
		\end{align}
		where  $c_0$ is the product $\prod(1-\frac{1}{\|P\|^2})$ over \textit{all} primes $P$ except $P=t$, so \begin{equation}\label{c_0}c_0=(1-\frac1q)/(1-\frac1{q^2})=\frac1{1+1/q}.\end{equation}
		
		
		Moreover, $\displaystyle\sum_{P\in\mathcal{P}^{\ge m_1}}\frac1{\|P\|^2}\le\sum_{j=m_1}^{\infty}\frac1{jq^j}=O\left(\frac1{m_1q^{m_1}}\right)$. This is by the classical bound (see~\cite[Proposition 2.1]{rosen2002number}): 
\begin{equation}\label{thm: ppt}
\pi_q(n) \le \dfrac{q^n}n. 
\end{equation}
 We replace the error term in~\eqref{eq: Uinfty} with $O\left(\frac1{m_1q^{m_1}}\right)$.
		
		It will thus suffice to bound $U(\infty,m_1) - U(r,m_1)$. Let us denote for any non-negative integer $j$, $v_j = \sum_{D \mid \mathcal{P}(m_1), \omega(D) = j, t\nmid D}\frac{1}{\|D\|^2}$. Note that $v_j$ is the $j$-th elementary symmetric polynomial of the finite multiset $\left\{\frac{1}{\|P\|^2} : P \in \PR^{<m_1}\setminus\{t\}\right\}$, whose elements are positive real numbers.
		
		Now, $v_j\le\frac{v_1^j}{j!}$, as by expanding $v_1^j$ we obtain $j!$ copies of $v_j$ and some other terms. Furthermore $v_1$ is a partial sum of  $\sum_{P\in\mathcal{P}}\frac1{\|P\|^2}<\infty$, hence $v_1 = O(1)$.
		Suppose $r = \alpha v_1$, for some $\alpha$. Then, for large enough $\alpha$ (e.g. $\alpha>2$),
		\[
		|U(\infty,m_1) - U(r,m_1)| = \left|\sum_{j=r+1}^{\infty} (-1)^j v_j\right| \le  \sum_{j=r+1}^{\infty} v_j \le \sum_{j=r+1}^{\infty} \frac{v_1^j}{j!} \le \sum_{j=r+1}^{\infty} \frac{v_1^r}{r!}\alpha^{r-j}\le\frac{v_1^r}{r!}.\]
For large enough $r<m_1/2$, we will achieve $\alpha>2$ since $v_1$ is bounded (and $m_1\to\infty$). Stirling's approximation tells us that $\frac{v_1^r}{r!}\asymp \left(\frac{v_1e}{r}\right)^r(2\pi r)^{-1/2}$. 
Both $U(r,m_1)$ and $U(r-1,m_1)$ are thus within $O((\gamma r)^{-r})$ of $U(\infty,m_1)$, where $\gamma$ is chosen as a lower bound for $1/(v_1e)$ (note that $(\gamma r)^{-r}$ decreases with $\gamma$). So
		$N_r = (1-\Lambda)c_0(1+O(1/m_1q^{m_1}) + O((\gamma r)^{-r}))$, and hence $N_r'$ is equal to
\begin{align*}
(1-\Lambda)c_0\Bigl(1+O\bigl(\tfrac1{m_1q^{m_1}}\bigr) + O((\gamma r)^{-r})\Bigr) &+ \sum_{{\substack{D \mid \mathcal{P}(m_1)\\ \omega(D) \le r\\ t\nmid D}}}\mu_\mathrm{Mob}(D)\mathrm{Err}_{\overline{\mu}}(n,D) \\
&+\Lambda\sum_{{\substack{D \mid \mathcal{P}(m_1)\\ \omega(D) \le r\\ t\mid D}}}\mu_\mathrm{Mob}(D)\mathrm{Err}_{\overline{\mu}}(n-2,D/t).
\end{align*}
		By definition of $\Lambda$, $c_0$, and $\kappa_{\overline{\mu}}$ from Theorem~\ref{thm: general measure squarefree}, we have $\kappa_{\overline{\mu}}=c_0(1-\Lambda)$. 
		
		To bound the remaining error: recall $\mathrm{Err}_{\overline{\mu}}(n,D)=\mathbb{P}_{\overline{\mu}}(D^2\mid f)-1/\|D\|^2$ for $t\nmid D$.
	Let $\deg D=k$, so $1/\|D\|^2=1/q^{2k}$. Then $|\mathrm{Err}_{\overline{\mu}}(n,D)|\le\sum_i q^{-i}\cf^{\lfloor \frac n{2k-i}\rfloor}$, where $i$ ranges over sums of degrees of prime factors of $D$ (each with multiplicity up to $2$), and there are $3^{\omega(D)}-1$ (almost $\tau(D^2)$) such terms.
		As $\cf<1$ we have that $q^{-i}\cf^{\lfloor \frac n{2k-i}\rfloor}$ is decreasing with $i$, so this sum is bounded by $3^{\omega(D)}\cf^{\lfloor \frac n{2k}\rfloor}$.
		Thus $\mathbb{P}_{\overline{\mu}}(D^2\mid f)=\frac1{\|D\|^2}+O(3^{\omega(D)}\cf^{\lfloor \frac n{2k}\rfloor})$. The formula for $t\mid D$ is similar.
		
		
		Now we estimate the sum of the $\mathrm{Err}_{\overline{\mu}}$ terms. Note that since $\cf<1$ and $k\le rm_1$, then $\cf^{\lfloor \frac n{2k}\rfloor}\le \cf^{\lfloor \frac n{2rm_1}\rfloor}$, so a bound on the total error is $\cf^{\lfloor \frac n{2rm_1}\rfloor}\sum_{{\substack{D \mid \mathcal{P}(m_1)\\ \omega(D) \le r}}}\tau(D^2)$.
		Across all $D$ we sum over, we have $\tau(D^2)\le 3^r$.
		The number of summands is $\binom{|\PR^{\le m_1}|}{r}+\dots+\binom{|\PR^{\le m_1}|}{0}$, and 
 as $r<m_1/2$, we can bound the sum by $(r+1)\binom{|\PR^{\le m_1}|}{r}\le |\PR^{\le m_1}|^r$. By~\eqref{thm: ppt}, $|\PR^{\le m_1}|\le\sum_{k\le m_1}\frac{q^{k}}{k}=O\Bigl(\frac{q^{m_1}}{m_1}\Bigr)$.
		The accumulated error here is thus 
		\begin{equation}\label{eq: N' error}
			O\left[\Bigl(\frac{q^{m_1}}{m_1}\Bigr)^r\cdot 3^r\cdot \cf^{n/2rm_1}\right].
		\end{equation}
Summing the errors and setting $r=r(n)\to\infty$ with $r<m_1/2$ gives the result.
\end{proof}
We have finished bounding $N'$, and now we are interested in the large prime regime. We use a H{\"o}lder argument, similar to~\cite[Prop. 3.3]{universality}, to prove:
\begin{lemma}\label{lemma: halasz-bound modified}
	Let $f(t)=\sum_{0\le i \le K} \varepsilon_i t^{i}$ be a polynomial of degree at most $K$ in $\mathbb{F}_q[t]$, with $\varepsilon_i$ chosen independently at random according to probability measures $\mu_0,\mu_1,\dots,\mu_{K}$ on $\FF_q$ satisfying $\cf=\max_{0\le i\le K} \max_{\alpha \in \mathbb{F}_q^\times} |\hat{\mu}_i(\alpha)|<1$.
	Then given a prime polynomial $P\ne t$ of degree $d\le K+1$,
	\begin{equation*}
		\P[f\equiv A\!\!\!\pmod{P}]\leq \Bigl(\frac{q-1+M}{q}\Bigr)^{\deg P},
	\end{equation*}
	for any polynomial $A$, and some $M<1$ dependent on $K$, $q$ and the measures $\mu_i$.
\end{lemma}
\begin{remark}\label{remark: M def}
	$M=\cf^{\lfloor\frac {K+1}d\rfloor}$, where $\cf$ is as above (first defined in Theorem~\ref{thm: general measure squarefree}). Since $K+1\ge d$, we have $M\le\cf$.
\end{remark}
\begin{proof} 
	There are $q^{\deg P}$ possible remainders modulo $P$. Since $f\equiv A\!\pmod{P}\implies P\mid f-A$, and $A$ is deterministic, the change in coefficients from subtracting $A\pmod P$ is equivalent to a translation of all the probabilities of elements of $\FF_q$, giving new probability measures $\mu_i'$. So we assume without loss of generality that $A=0$. Now
\begin{equation}\label{eq: P mid f}\P[P\mid f]=\E(\mathbbm{1}_{P\mid f})=\frac1{\|P\|}\sum_{\deg a<\deg P}\E\left(\textbf{e}_q\Bigl( \frac{af}P\Bigr)\right).\end{equation}
	
	The coefficients of $f$ are the $\varepsilon_{i}\sim\mu_i$. By additivity of the character $\textbf{e}_q$, we can rewrite each summand $\textbf{e}_q( \frac{af}P)$ as a product $\displaystyle\prod_{i=0}^{K} \textbf{e}_q\left( \dfrac aP\varepsilon_it^i\right)$. 
	Since all $\varepsilon_i$ are independent, the expectation of our product is a product of expectations. 
	Additionally, writing $a=\sum_{j=0}^{\deg P-1}a_jT^j$ lets us sum over possible vectors $\vec{a}\in\mathbb{F}_q^{\deg P}$ of the $a_j$. Thus, 
	\begin{equation}\label{eq: probabilistic expression P|f}
		\P[P\mid f]=\frac{1}{q^{\deg P}}\sum_{\vec{a}}\prod_{i=0}^{K} \E\left[X(\varepsilon_i,\vec{a})\right],\quad\text{where } X(\varepsilon_i,\vec{a})=\textbf{e}_q\Bigl(\dfrac {(\sum a_lt^l)\varepsilon_it^i} P\Bigr).
	\end{equation} 
	We now apply the generalised H\"older's inequality; for $j\in\mathbb{N}$ and $p_1,\dots,p_j$ such that $\frac1{p_1}+\dots+\frac1{p_j}=1$, we have $\|f_0\dots f_{j-1}\|_1\le \|f_0\|_{p_1}\cdot\dots\cdot\|f_{j-1}\|_{p_j}$. In this scenario we pick the value of $j$ to be $\lfloor\frac {K+1}d\rfloor$ where $d$ as given in the statement, and all $p_k$ equal to $\lfloor\frac {K+1}d\rfloor$ as well. 
	Note that $|\textbf{e}_q\left( \dfrac aP\varepsilon_it^i\right)|=1$ for all $i$, so to get an upper bound overall, we can discard the factors with	$i\ge jd$. Set $f_k(\vec{a})=\prod_{h=0}^{d-1} \E\left[X(\varepsilon_{dk+h},\vec{a})\right]$. 	
	We have
	\begin{equation}\label{eq: holder first bound}
		\begin{split}
			&   \sum_{\vec{a}}\prod_{i=0}^{K} |\E\left[X(\varepsilon_i,\vec{a})\right]|\le \sum_{\vec{a}}\prod_{k=0}^{j-1}|f_k(\vec{a})|=\|f_0(\vec{a})\dots f_{j-1}(\vec{a})\|_1 \\ 
			&\leq \|f_0(\vec{a})\|_{j}\cdot\dots\cdot\|f_{j-1}(\vec{a})\|_{j}
			= \prod_{k=0}^{j-1}\left(\sum_{\vec{a}} \prod_{h=0}^{d-1}|
			\E\left[X(\varepsilon_{dk+h},\vec{a})\right]|^{j}\right)^{\frac{1}{j}}
			\\&\leq \max_k\sum_{\vec{a}} \prod_{h=0}^{d-1}\left|\E\left[X(\varepsilon_{dk+h},\vec{a})\right]\right|^j.
		\end{split}
	\end{equation}
	Recall that 
$\textbf{e}_q(x)=\psi(x_{-1})$ for $\psi(x)=\exp(2\pi i\mathrm{tr}(x)/p)$.
	To calculate the relevant coefficients of $ X(\varepsilon_{dk+h},\vec{a})$, for each fixed block we need $2d-1$ coefficients from the Laurent expansion of $1/P$: 
as the polynomial $a$ ranges over elements of $\mathbb{F}_q[t]$ of degree less than $\deg P$, it can be expressed in $d$ coefficients, and so for each of the $d$ random $\varepsilon$ we need a block of $d$ coefficients from $1/P$.
	But by Claim~\ref{claim: d zeroes}, there are not $d$ consecutive zeroes in the expansion of $a/P$ unless $a=0$, and hence the mapping of the coefficients $\vec{a}\mapsto\sum_{j=0}^{d-1} \left[a_j \cdot(\frac1P)_{-j-i-1}\right]_{i=dk}^{dk+d-1}$ is injective.
	Thus we may rephrase our bound in~\eqref{eq: holder first bound} to 
	\[\max_k\sum_{\vec{b}} \prod_{h=0}^{d-1}\left|\E\left[\psi(b_{h}\varepsilon_{dk+h})\right]\right|^{j}\]
	As $\vec{b}$ ranges over all $\mathbb{F}_q^d$ and everything is independent, this equals
	\begin{equation}\label{eq: holder final ver}
		\max_k \prod_{h=0}^{d-1}\sum_{b\in\mathbb{F}_q}\left|\E\left[\psi(b\varepsilon_{dk+h})\right]\right|^{j}
	\end{equation} 
	By definition of the expectation with respect to  $\mu_{dk+h}$,
	\begin{equation}\label{expectation eq} \E[\psi(b\varepsilon_{dk+h})] = \sum_{\alpha\in\FF_q}\mu_{dk+h}(\alpha)\psi(b\alpha) = \hat{\mu}_{dk+h}(-b).\end{equation}
	
	Because $\mu_{dk+h}$ is a probability measure, the absolute value of the expectation is bounded by $1$, with equality occurring trivially when $b=0$ (yielding $\hat{\mu}(0) = 1$). For non-zero frequencies $b$, if the measure is not entirely supported on a single coset of the trace kernel, the magnitude is strictly bounded away from $1$. The global maximum amplitude for such frequencies is $\cf$, which we already introduced.
	
	The bound $\cf$ can be applied to at least one term as by `nondegeneracy' of the trace function, there exists a $b_0\in\FF_q$ and $\alpha_1,\alpha_2\in\operatorname{supp}(\mu_{dk+h})$ such that $\operatorname{arg}(\psi(b_0\alpha_1))\ne\operatorname{arg}(\psi(b_0\alpha_2))$. For $b=b_0$, at least one of the summands in the sum of~\eqref{expectation eq} cannot be written as a nonnegative real scalar multiple of the others. Letting $w=\min_i\min_{\gamma\in\text{supp}(\mu_i)}\mu_i(\gamma)\le \frac12$ and using $e^{i\theta}=\cos\theta +i\sin\theta$, we calculate a bound on the magnitude using Pythagoras' theorem:
	\begin{equation}\label{eq: C def}
		\cf \le \max_i\!\!\! \max_{\substack{b \neq 0\\\arg(\psi(b\alpha))\text{ nonconstant}\\\text{for }\alpha\in\operatorname{supp}(\mu_i)}}\!\!\!\!\!\!\!\!\! |\hat{\mu}_i(-b)|\le \sqrt{(1-w+w\cos(2\pi/p))^2+w^2\sin^2(2\pi/p)} < 1.
	\end{equation}
	We may write
	\begin{equation}\label{eq: M ineq}
		\sum_{b\in\mathbb{F}_q}\left|\E\left[\psi(b\varepsilon_{dk+h})\right]\right|^{j} \le q-1+\cf^{j}.
	\end{equation}
	
	To conclude the proof, we combine \eqref{eq: probabilistic expression P|f}, \eqref{eq: holder final ver} and \eqref{eq: M ineq}, with $M=\cf^{j}=\cf^{\lfloor\frac {K+1}d\rfloor}$.
\end{proof}

\begin{remark}
	The bound can be tightened using Parseval's identity for general probability measures. Consider the discrete Fourier transform of the probability measure $\mu_{dk+h}$, given by $\hat{\mu}_{dk+h}(-b) = \sum_{x\in\FF_q} \mu_{dk+h}(x) \psi(bx)$.
	By Parseval's identity, the sum of the squares of the Fourier magnitudes is proportional to the $L^2$ norm of the measure:
	$$\sum_{b\in\FF_q}|\hat{\mu}_{dk+h}(-b)|^2 = q\sum_{x\in\FF_q}\mu_{dk+h}(x)^2=:V_{dk+h};$$
	Here $V_{dk+h}$ acts as a measure of how ``concentrated'' the probability distribution is.
	This allows us to improve the bound in~\eqref{eq: M ineq} by splitting the sum:
	\begin{itemize}
		\item For some $b$ we have $|\E\left[\psi(b\varepsilon_{dk+h})\right]| = 1$. By Parseval, the sum of squares of only these terms is $\le V_{dk+h}$, but each is $1$, so the number of such $b$ is bounded by $V_{dk+h}$.
		\item For the remaining values of $b$ we get $|\E\left[\psi(b\varepsilon_{dk+h})\right]| \le \cf < 1$. Thus
		$$\sum_{\substack{b\\|\E\left[\psi(b\varepsilon_{dk+h})\right]| \le \cf}} \left|\E\left[\psi(b\varepsilon_{dk+h})\right]\right|^{j} \le \cf^{j-2} \sum_{\substack{b\\|\E\left[\psi(b\varepsilon_{dk+h})\right]| \le \cf}} \left|\E\left[\psi(b\varepsilon_{dk+h})\right]\right|^2 \le \cf^{j-2} V_{dk+h}.$$
	\end{itemize} 
(We may assume $j\ge 2$.)	Combining these gives: 
	\begin{equation}\label{eq: M ineq improved gen}
		\sum_{b\in\mathbb{F}_q}\left|\E\left[\psi(b\varepsilon_{dk+h})\right]\right|^{j} \le V_{dk+h}(1 + \cf^{\lfloor\frac {K+1}d\rfloor-2}).
	\end{equation}
	For the uniform measure supported on a set of size $q-s$ we get $V_{dk+h} = \frac{q}{q-s}$, and in this case the Parseval method can give a better bound than the original, especially for large $j$. For now we proceed using the (cruder) bound stated in the theorem.
\end{remark} 
\begin{lemma}\label{lem: large primes N''}
	Let $N''$ be the event that $f$ is divisible by the square of at least one prime $P$ with $\deg P \ge m_2$. If $m_2=o(n/p)$, then
	\[ \mathbb{P}_{\overline{\mu}}(N'') \le \mathbb{P}_{\overline{\mu}}(f' = 0) + \frac{n-1}{m_2} \tilde{\cf}^{m_2}, \]
	where $\tilde{\cf} = \frac{q-1+\cf}{q} < 1$.
\end{lemma}
\begin{proof}
		We bound the probability that $f$ is divisible by the square of a ``large'' prime (degree at least $m_2$). Recall that $N'' = \bigcup_{P \in \mathcal{P}^{\ge m_2}} \{f \in \mathcal{M}_n : P^2 \mid f\}$.
		
		We start by writing:
		$\displaystyle\mathbb{P}_{\overline{\mu}}(N'') \le \mathbb{P}_{\overline{\mu}}(f' = 0) +\sum_{P \in \mathcal{P}^{\ge m_2}} \mathbb{P}_{\overline{\mu}}(P^2 \mid f).$

		We make use of $P^2 \mid f$ being equivalent to $P$ dividing both $f$ and $f'$. We apply the law of total probability, partitioning our polynomial space by the value of the formal derivative. Let $T_n$ be the set of all possible derivatives of polynomials in $\mathcal{M}_n$. Then:
		\begin{equation}\label{eq: N''}\mathbb{P}_{\overline{\mu}}(N'') \le\mathbb{P}_{\overline{\mu}}(f' = 0)+ \sum_{d \in T_n\setminus\{0\}} \mathbb{P}_{\overline{\mu}}(f' = d) \sum_{\substack{P \in \mathcal{P}^{\ge m_2} \\ P \mid d}} \mathbb{P}_{\overline{\mu}}(P \text{ divides } f \mid f' = d).\end{equation}

We bound the conditional probability $\mathbb{P}_{\overline{\mu}}(P \text{ divides } f \mid f' = d)$.	Note that if $f'=d$ then there exist coefficients $a_i$ such that $f=f_0+\sum_{i\le n/p} a_it^{pi}$, where $f_0$ is some fixed polynomial with derivative $d$. Since $f\equiv 0\pmod P$ and we are over the finite field $\mathbb{F}_q$ of characteristic $p$, we have $(\sum x_i)^p=\sum x_i^p$, so $f_0+(\sum_{i\le n/p} a_i^{q/p}t^{i})^p\equiv 0\pmod P$. Hence $\sum_{i\le n/p} a_i^{q/p}t^{i}\equiv -f_0^{q^{\deg P}/p}\pmod P$, as the norm of $P$ is $q^{\deg P}$. The right hand side here is deterministically fixed, and the $a_i$ are randomly sampled according to probability measures. 
		
		The condition $P \mid f$ thus reduces to a congruence modulo $P$ acting solely on these remaining free $p$-th powers,
and we reduce our problem to estimating
		\begin{equation}\label{eq: frob phrasing}
			\mathbb{P}_{\overline{\mu}}\Bigl(\sum_{i\le n/p} a_i^{q/p}t^{i}\equiv (-f_0)^{q^{\deg P}/p}\pmod P\Bigr).
		\end{equation}		
Since $d$ is a fixed polynomial of degree at most $n-1$ in the inner sum in~\eqref{eq: N''}, the number of its prime factors with degree at least $m_2$ cannot exceed $\frac{n-1}{m_2}$.
	Combining with an estimate of~\eqref{eq: frob phrasing} gives a constant bound for the inner sum in~\eqref{eq: N''}, and we can multiply this to the outermost sum $\sum_{d \in T_n} \mathbb{P}_{\overline{\mu}}(f' = d)\le 1$.
		
To bound~\eqref{eq: frob phrasing}, we apply Lemma~\ref{lemma: halasz-bound modified}, using the degree $K = \lfloor n-1/p \rfloor$; note that we are now working with a modified set of measures which give random coefficients $a_i^{q/p}$ rather than $a_i$, but the bounds $M$ and $\tilde{\cf}$ will still apply. For primes of degree up to $K+1$, we get the bound $(\frac{q-1+M}{q})^{\deg P}$, but the fraction can be bounded by $\tilde{\cf}$ for any $P$, so we replace with $\tilde{\cf}^{\deg P}$. 
For primes $P$	such that $\deg P > K+1$ where the Lemma does not apply, the congruence in~\eqref{eq: frob phrasing} becomes an equality in $\mathbb{F}_q[t]$ after reducing modulo $P$. The probability of all random coefficients matching the fixed right hand side is bounded by the product of their maximum point masses, say $D_{\max}^{K+1}$, where $D_{\max}\le\frac{1+(q-1)\cf}q\le\tilde{\cf}$. If $m_2$ is a smaller order of magnitude than $n/p$ and equivalently $K$, the bound $\tilde{\cf}^{m_2}$ is unconditional across all degrees, which completes the proof.
\end{proof}
\subsection{Proof of Theorem~\ref{thm: general measure squarefree}}\label{subsection: proof thm 1.2}
Our starting point is~\eqref{eq: sec 3 start}. By Lemma~\ref{lem: estimating N'},
			\[ \mathbb{P}_{\overline{\mu}}(N') = \kappa_{\overline{\mu}}1 + O\left(\frac{1}{m_1q^{m_1}}\right) + O((\gamma r)^{-r}) + O\left( \left(\frac{q^{m_1}}{m_1}\right)^r \cdot 3^r \cdot \cf^{\frac{n}{2rm_1}} \right); \]
			here $\kappa_{\overline{\mu}} = c_0(1-\Lambda)$ is the asymptotic density, with $c_0$ defined in~\eqref{c_0} and $\Lambda$ defined in~\eqref{prob.tsquared}, and $\cf\le\cs < 1$ is the maximum Fourier amplitude of the measures.
			
Let $\lambda > 0$ be a small constant dependent only on $q$ and $\cs$, and set
\begin{equation}\label{eq: m,r choice} m_1 = m_2 = m = \left\lceil \lambda \frac{n^{1/4} \sqrt{\log n}}{p^{1/2} (\log q)^{1/4}} \right\rceil, \quad r = \left\lfloor \lambda \frac{n^{1/4}}{p^{1/2} (\log q)^{1/4} \sqrt{\log n}} \right\rfloor. \end{equation}
Specifically, we impose $\lambda^4 < \frac{1}{2} p^2 \log(1/\cs)$; this ensures $rm\log q<\frac{n}{2rm}\log (1/\cs)$, which is sufficient for the argument below. 
\begin{remark}
Since $p$ and $q$ are fixed constants in this section, $m \asymp_q n^{1/4} \sqrt{\log n}$ and $r \asymp_q n^{1/4} / \sqrt{\log n}$. Our chosen $m$, $r$ extend to Section~\ref{sec: vary all}, when $p$ and $q$ vary.\end{remark}
We calculate the asymptotics of the error terms for $N'$:
			\begin{enumerate}
				\item By definition, $\frac{1}{m_1q^{m_1}}\ll\exp(-m_1 \log q)\ll \exp(-n^{1/4} \sqrt{\log n} c_q)\ll_{q,\cs}q^{-n^{1/4}}$, some $c_{q,\cs}>0$.
				\item Since $r\log r\asymp_{q,\cs} \frac{\lambda n^{1/4}}{\sqrt{\log n}} \left(\frac{1}{4} \log n + O_q(\log \log n)\right)$, we have \newline $O((\gamma r)^{-r})~=~O(\exp(-r \log (\gamma r)))~=~O_{q,\cs}(q^{-n^{1/4}})$.
				\item The exponential sum error is $O_{q,\cs}\left(\Bigl(\frac{q^{m_1}}{m_1}\Bigr)^r \cdot 3^r \cdot \cf^{\frac{n}{2rm_1}}\right)$. Consider its logarithm:
				\[ r m_1 \log q + r(\log 3 - \log m_1) - \frac{n}{2rm_1}\log(1/\cf). \]
				Dropping the negative middle term and using $(rm)^2\sim\frac{\lambda^4n}{p^2\log q}$ gives an upper bound of:
				\[ rm\log q\left(1-\frac{p^2}{2\lambda^4}\log(1/\cf)\right). \]
				By our choice of $\lambda$, the bracketed constant is negative, and $rm\asymp_{q,\cs}\sqrt{n}$ so this error decays much faster than $O_{q,\cs}(q^{-n^{1/4}})$.
			\end{enumerate}
			
			Finally, applying Lemma~\ref{lem: large primes N''}, we bound the large prime contribution by:
			\[ \mathbb{P}_{\overline{\mu}}(N'') \le \mathbb{P}_{\overline{\mu}}(f' = 0) + \frac{n-1}{m_2} \tilde{\cf}^{m_2}. \]
			If $f' = 0$ then $f$ is a perfect $p$-th power, meaning all $n-1 - \lfloor n-1/p \rfloor$ coefficients not associated with $p$-th powers are exactly zero. This occurs with probability $\le D_0^{n(1-1/p)}$, for $D_0 = \max_i \mu_i(0)$. Since $\mu_i(0)\le\frac1q\sum_{\alpha\in\FF_q}|\hat{\mu_i}(\alpha)|\le\frac{1+(q-1)\cf}{q}<1$, we get 
			 $\mathbb{P}_{\overline{\mu}}(f' = 0)=O_{q,\cs}(q^{-n^{1/4}})$. (If $p\nmid n$, the expression is 0.)
			
			For the remaining term, we take the logarithm:
			\[ \log(n-1) - \log m_2 - m_2 \log(1/\tilde{\cf}). \]
			Substituting $m_2 \asymp_q n^{1/4} \sqrt{\log n}$, the term $-m_2 \log(1/\tilde{\cf})$ dominates. As $\tilde{\cf} < 1$ is constant, the probability decays asymptotically as $\exp(-c_qn^{1/4} \sqrt{\log n} \log(1/\tilde{\cf}))=O_{q,\cs}(q^{-n^{1/4}})$.
			
			Both the sieve error and the large prime error are $O_{q,\cs}(q^{-n^{1/4}})$, so $\mathbb{P}_{\overline{\mu}}(f \text{ is squarefree}) = \kappa_{\overline{\mu}} + O_{q,\cs}(q^{-n^{1/4}})$.	
\qed
\subsection{Varying $n$ and $q$ simultaneously}\label{sec: vary all}

In Theorem~\ref{thm: general measure squarefree}, $q$ is fixed while
$n\to\infty$; we consider what happens when $n$ and $q$ to tend to infinity simultaneously.

The main difficulty is that the uniform bound on the non-trivial Fourier
coefficients tends to $1$ as the characteristic $p$ grows. We require a minimum rate of growth for $n$ relative to $q$ and $p$ to ensure the error terms remain negligible. We may use an improvement on~\eqref{eq: M ineq improved gen} of Lemma~\ref{lemma: halasz-bound modified} to obtain the tightest bounds, and make small tweaks to the proof of~\ref{thm: general measure squarefree}, with the same bounds on $r$, $m_1$, $m_2$ as before to get the following more general result:
\begin{theorem}\label{thm: simultaneous_limit}
	There exist absolute constants $C,c'>0$ with the following property:
	Let $q=p^k\ge 3$ be a prime power and let $n\ge 2$. Let
	$\overline{\mu}=(\mu_0,\dots,\mu_{n-1})$
	be probability measures on $\mathbb{F}_q$, and suppose that, for every
	$0\le i<n$, the support of $\mu_i$ is not contained in a proper affine
	hyperplane of $\mathbb{F}_q$. 
	Define
	\[
	w=w(\overline{\mu})
	:=\min_{0\le i<n}\min_{x\in\operatorname{supp}(\mu_i)}\mu_i(x).
	\]
	Suppose that
	\begin{equation}\label{eq: simultaneous growth}
		n\ge		C
		\frac{qp^5\log q}{w^2}
		\left(
		\log q+\log\frac1w
		\right).
	\end{equation}
	Then a random monic polynomial $
	f(t)=t^n+\sum_{i=0}^{n-1}\varepsilon_it^i$ with $\varepsilon_i\sim\mu_i$ independent coefficients, is squarefree with probability
	\[
	\mathbb{P}_{\overline{\mu}}(f\text{ is squarefree})=\frac{1}{1+1/q}\left(1-\mu_0(0)\mu_1(0)\right)+	O(E(n,p,q,w)),
	\]
	where
	\begin{equation}\label{eq: simultaneous error}
		E(n,p,q,w)=
		\exp\left(
		-c'
		\frac{w^{1/4}n^{1/4}\sqrt{\log n}}{p^{1/2}(\log q)^{1/4}}
		\right)+			\exp\left(	-c'	\frac{w^2n}{qp^5\log q}\right).
	\end{equation}
	In particular, along any sequence of $(n,q,w)$ where $q\to\infty$ and
	\eqref{eq: simultaneous growth} holds, we get $E(n,p,q,w)\to~0$.
\end{theorem}
Our proof of Theorem~\ref{thm: simultaneous_limit}, which we omit for reasons of length, consists of verifying the new error term in \eqref{eq: simultaneous error}.
		\section{Squarefrees with restricted coefficients}\label{section: ban more}
		We can now recover the results for restricted coefficient sets mentioned in the introduction as a direct consequence of Theorem~\ref{thm: general measure squarefree}.
\begin{corollary}\label{lem: affine subfield}
			Let $A_0, \dots, A_{n-1} \subseteq \mathbb{F}_q$, such that none of them are contained within a proper affine subspace of $\mathbb{F}_q$ viewed as a vector space over $\mathbb{F}_p$ (the characteristic field). 
			
			The number of squarefree, monic polynomials with coefficients restricted to these sets is:
			\[ \kappa \left(\prod_{i=0}^{n-1} |A_i|\right) (1+O_q(q^{-n^{1/4}})), \] where $\kappa=\frac1{1+1/q}(1-\frac{\mathbbm{1}\{0\in A_0\}}{|A_0|}\frac{\mathbbm{1}\{0\in A_1\}}{|A_1|})$.
		\end{corollary}
		\begin{proof}
			Let $\mu_i$ be the uniform probability measure on $A_i$. The Fourier transform for $\alpha \in \mathbb{F}_q^\times$ is: 
			\[ \hat{\mu}_i(\alpha) = \frac{1}{|A_i|} \sum_{x \in A_i} \psi(-\alpha x). \]
	By definition, $\psi(-\alpha x)=\exp(2\pi i\mathrm{tr}(-\alpha x)/p)$, and it always has magnitude $1$; there are $|A_i|$ summands.
			
If $|\hat{\mu}_i(\alpha)|=1$, then the complex arguments are all the same, so $\mathrm{tr}(-\alpha x)$ is constant over all $x \in A_i$; say, it is $c\in\FF_p$. Then $A_i$ is contained within the affine hyperplane defined by $\{x \in \mathbb{F}_q : \mathrm{tr}(-\alpha x) = c\}$. 

However, by our choice of $A_i$, this is impossible. 
Hence $\vert{}\hat{\mu}_i(\alpha)\vert{} < 1$; since $\mathbb{F}_q$ is finite, the maximum value of this magnitude over all $\alpha \in \mathbb{F}_q^\times$ is bounded strictly away from $1$. Thus, $\cf < 1$, Theorem~\ref{thm: general measure squarefree} is satisfied, and multiplying the resulting probability by the sample space size $\prod \vert{}A_i\vert{}$ completes the proof.			
		\end{proof}
		A special case which we care about is as follows: $q=p$ and $A_i=\{0,1\}$ for all $i$. This satisfies Corollary~\ref{lem: affine subfield} as $\FF_p$ has no proper affine subspaces of size $>1$. We then obtain Theorem~\ref{thm: main result}.

		Note that we wrote the polynomials are contained in $\FF_q[t]$; we can consider them as elements of $\FF_p[t]$ contained in this field extension, as this does not affect squarefreeness. In general, when all allowed coefficients are in $\FF_{p^k}\subseteq\FF_q$ but are not contained within a proper affine subspace of $\FF_{p^k}$, we get Corollary~\ref{thm: gen sec 4}.

We can make a comparison with Porritt's main theorems in~\cite{Porritt2019}; in them, he takes the number $s$ of forbidden digits to be at most $\sqrt{q}/2$. In our setting, since any proper affine subspace of $\FF_q$ has size at most $q/p$, we can extend to $s<q(p-1)/p$.
		\section{Working with large characteristic}\label{section: large char}
		All our results so far have relied on $q$, in particular $p$, being constant and thus small compared to $n$. 
We now consider the regime where the degree $n$ is fixed, and field size $q \to \infty$. In this $q$-limit, the probability of a random (unrestricted) polynomial being squarefree, which is $1 - 1/q$ for $n>1$, converges to $1$. 
		
		We wish to establish that this asymptotic behaviour persists even when almost all coefficients are deterministically fixed, provided the lowest-degree coefficients remain sufficiently free. In keeping with our previous style, we will obtain this as a special case of a measure-theoretic result.
		\begin{definition}	
			We define the ``maximum point mass'' of a probability measure $\mu$ on~$\mathbb{F}_q$:
			\[ D(\mu) = \max_{x \in \mathbb{F}_q} \mu(x). \]\end{definition}
		
		\begin{theorem}\label{thm: large q limit measure}
			Let $n \ge 2$ be a fixed integer. Let $\overline{\mu} = (\mu_0, \dots, \mu_{n-1})$ be a sequence of probability measures on $\mathbb{F}_q$. Let $f(t) = t^n + \sum_{i=0}^{n-1} a_i t^i$ be a random monic polynomial whose coefficients $a_i$ are drawn independently according to $\mu_i$. 
			
			Then the probability that $f$ is squarefree is bounded below by:
			\[ \mathbb{P}_{\overline{\mu}}(f \text{ is squarefree}) \ge 1 - D(\mu_1) - (n-1)D(\mu_0). \]
			In particular, if the distributions for the lowest two terms satisfy $D(\mu_0), D(\mu_1) \to 0$ as $q \to \infty$, then the probability that $f$ is squarefree tends to $1$, regardless of the restrictions placed on the higher-degree coefficients.
		\end{theorem}
		
		\begin{proof}
			We decompose the random polynomial as $f(t) = h(t) + a_1 t + a_0$, where $h(t) = t^n + \sum_{i=2}^{n-1} a_i t^i$. The coefficients of $h(t)$ are sampled according to the product measure $\overline{\mu}_{>1} = \mu_2 \times \dots \times \mu_{n-1}$. We proceed by conditioning on a fixed $h(t)$, bounding the failure probability uniformly over all such choices. 
			
			A polynomial $f(t) \in \mathbb{F}_q[t]$ is squarefree if and only if it is coprime to its formal derivative. For $f(t)$ to share a root with $f'(t)$, there must exist some root $\rho \in \overline{\mathbb{F}}_q$ such that $f'(\rho) = 0$ and $f(\rho) = 0$.
			
			The formal derivative is
			$f'(t) = h'(t) + a_1$, which is $0$ if and only if $h'(t)$ is a constant polynomial $-c$ and $a_1 = c$. Because $h(t)$ is fixed, there is at most one specific ``bad" value $c \in \mathbb{F}_q$ that can cause $f'(t)$ to be the $0$ polynomial. Note that $\mathbb{P}(a_1=c)\le D(\mu_1)$. 
			
			If $a_1$ avoids this value (with probability at least $1 - D(\mu_1)$), $f'(t)$ is guaranteed to be a non-zero polynomial of degree $d$, where $0 \le d \le n-1$. Thus, $f'(t)$ has exactly $d$ roots in the algebraic closure $\overline{\mathbb{F}}_q$, counting multiplicities. Let these roots be $\rho_1, \dots, \rho_d$; 			
			for $f(t)$ to not be squarefree, it must vanish at one of the $\rho_i$:
			\[ f(\rho_i) = h(\rho_i) + a_1 \rho_i + a_0 = 0. \]
			This imposes a strict linear constraint on $a_0$:
			\[ a_0 = -h(\rho_i) - a_1 \rho_i. \]
			
			Because $h(t)$ and $a_1$ are completely determined prior to the selection of $a_0$, the specific values $-h(\rho_i) - a_1 \rho_i$ are fixed constants. Consequently, there are at most $d \le n-1$ specific ``bad'' values in $\overline{\mathbb{F}}_q$ (and thus at most $n-1$ elements in $\mathbb{F}_q$) that $a_0$ could take to force $f(t)$ to share a root with $f'(t)$.
			
			By the union bound, the probability that $a_0$ is drawn as any of these $n-1$ specific elements is at most $(n-1)D(\mu_0)$. 
			
			Using the Law of Total Probability to integrate over all possible choices of $h(t)$, the total probability of generating a polynomial that is not squarefree is bounded above by the sum of the probabilities of these two cases of failure:
			\begin{align*}
				\mathbb{P}_{\overline{\mu}}(f \text{ is not squarefree}) &\le D(\mu_1) + (n-1)D(\mu_0). \\
				\implies \mathbb{P}_{\overline{\mu}}(f \text{ is squarefree}) &\ge 1 - D(\mu_1) - (n-1)D(\mu_0).\qedhere \end{align*}
		\end{proof}
		We may now state this in terms of allowed and restricted coefficients. Specifically, we will restrict the two lowest coefficients of the polynomial to be chosen from sets $S_0$ and $S_1$, while keeping the rest fixed. The coefficient is uniformly selected across the allowed set, so the maximum point mass for these coefficients is $1/|S_i|$. 
		Fixed coefficients (of higher powers) can be thought of as being sampled by a Dirac delta measure.
		\begin{corollary}\label{thm: large q limit} 
			Let $n \ge 2$ be a fixed integer. Let $h(t) \in \mathbb{F}_q[t]$ be a fixed polynomial of degree $n$. Let $S_0, S_1 \subseteq \mathbb{F}_q$ be sets of allowed coefficients for the constant and linear terms, respectively. 
			
			If $f(t) = h(t) + a_1 t + a_0$ is generated by choosing $a_1 \in S_1$ and $a_0 \in S_0$ uniformly and independently at random, then:
			\[ \mathbb{P}(f \text{ is squarefree}) \ge 1 - \frac{1}{|S_1|} - \frac{n-1}{|S_0|}. \]
			In particular, if $|S_0|, |S_1| \to \infty$ as $q \to \infty$, such polynomials are squarefree with probability tending to 1.
		\end{corollary}
In this format, our result is closely related to the work of Rudnick~\cite{rudnickvaryq}. We give a corollary of Rudnick's main theorem, applicable to our setting:

\begin{theorem}[{Special case of \cite[Theorem 1.2]{rudnickvaryq}}]\label{thm: rudnick special}
	Let $h(t) \in \mathbb{F}_q[t]$ be a fixed monic polynomial of degree $n \ge 2$. As $q \to \infty$, the proportion of elements $(a_0,a_1) \in \mathbb{F}_q^2$ such that the polynomial $h(t)+a_1t+a_0$ is squarefree in $\mathbb{F}_q[t]$ is given by:
	\[ \frac{\#\{(a_0,a_1)\in \mathbb{F}_q^2 : h(t) +a_1t+a_0 \text{ is squarefree} \}}{q^2} = 1 + O\left(\frac{n}{q}\right). \]
\end{theorem}
We prove this from \cite[Theorem 1.2]{rudnickvaryq} by taking $f(t,x)=h(t)-t^2+x$, where $x$ is an arbitrary monic polynomial in $t$ of degree $2$, which is $t^2+a_1t+a_0$. This gives separability in $x$, squarefree content as $1$ is coprime to $h(t)-t^2$, bounded degree and a height of at most $\max(\deg_t(h),2)=n$.
Corollary~\ref{thm: large q limit} can be deduced by modifying Rudnick's proof methods, but the approach is more complicated than the elementary probability proof presented above.

Corollary~\ref{thm: large q limit} also improves the following result of Oppenheim and Shusterman~\cite{oppshu}:
		\begin{corollary}[{\cite[Corollary 1.4]{oppshu}}] \label{cor}
			
			Fix an integer $n\ge 2$.
			For every finite field $\mathbb{F}$ pick subsets $S_0 = S_0(\mathbb{F}), \dots, S_n = S_n(\mathbb{F}) \subseteq \mathbb{F}$ having the same cardinality $C(\mathbb{F})$ such that 
			$\lim_{|\mathbb{F}| \to \infty} C(\mathbb{F}) = \infty.$
			Then
			\begin{equation*}
				\lim_{|\mathbb{F}| \to \infty} \frac{\# \big\{(a_0, \dots, a_n) \in S_0 \times \dots \times S_n 
					: \sum_{i=0}^n a_iT^i \ \text{is squarefree} \big\}}
				{\# (S_0 \times \dots \times S_n)} = 1.
			\end{equation*}
		\end{corollary}
		Notice that here, we require a choice for every coefficient, rather than the lowest two. Oppenheim and Shusterman also provide some results with most coefficients fixed in their paper, which are about existence of squarefree polynomials in the resulting ``box'' set $\prod S_i$ rather than a quantitative result.
\section{What happens when the Fourier amplitude is 1?}\label{section: c=1}
In sections \ref{section: restrict coeff main}, \ref{section: ban more} and \ref{sec: vary all}, the probability measures we studied were supported on at least two elements; this is a necessary (not sufficient) condition to ensure $\cf<1$. We consider an example where this is not enforced.

We use our previous notation; the random polynomial is $f=t^n+\sum_if_it^i\in\mathcal{M}_n,$ where $f_i$ are sampled according to probability measures $\mu_i$. (Assume $n\ge 2$.) Let $\overline{\mu} = (\mu_0, \dots, \mu_{n-1})$, where $\mu_i$ are defined as follows:
\begin{itemize}
	\item If $p\nmid i$, then let $\mu_i(0)=1$, and let $\mu_i(\alpha)=0$ for all $\alpha\ne 0$.
	\item If $p\mid i$, let $\mu_i$ be the uniform measure on all of $\FF_q$.
\end{itemize}
If the degree $n$ is a multiple of $p$, then using the property $\sum x_i^p=(\sum x_i)^p$, we may write $f=(t^{n/p}+\sum_{j=0}^{\lfloor (n-1)/p \rfloor}\tilde{f}_jt^j)^p$. This is a $p$-th power, and in particular divisible by a square, so the probability of being squarefree is $0$.

If $n$ is not a multiple of $p$, we use the fact that $f$ being squarefree is equivalent to $\gcd(f,f')=1$. In our case, $f'=nt^{n-1}$. The only prime factor of this polynomial is $t$, hence $f$ is squarefree if and only if it is not a multiple of $t^2$. By our definition of $\overline{\mu}$, this has probability $1-\mu_0(0)=(q-1)/q$.

Since the probability alternates between these two values forever, there is no asymptotic density, but both $0$ and $(q-1)/q$ are distinct from the value of $\dfrac{1-1/q}{1+1/q}$ quoted by Theorem~\ref{thm: general measure squarefree}.
\subsection{A one-dimensional example over $\FF_4$}
Let $\FF_4=\FF_2(\omega)$, where $\omega^2+\omega+1=0$, and let
$V=\{0,\omega\}$. Then $V$ is an $\FF_2$-subspace of
$\FF_4$ which is not a subfield -- the smallest possible example of a subspace of $\FF_q$ which is not a subfield. Let $S_n(V)$ be the set of monic polynomials of degree $n$ whose non-leading coefficients belong to $V$, equipped with the uniform probability measure.

\begin{theorem}\label{thm: F4 subspace counterexample}
	As $n\to\infty$,
	\[\mathbb{P}_{f\in S_n(V)}(f\text{ is squarefree})
	=\kappa_4+O\left(\frac1n\right),	\]
	where  $I_2(d)$ is the number of degree $d$ monic irreducible polynomials in $\FF_2[t]$ and
	\[\kappa_4=\frac34\prod_{k\ge1}	\left(1-\frac{2}{2^{4k}}\right)^{I_2(2k)}\ge\frac58>\frac35.\]
\end{theorem}
In particular, this constant differs from the main-term expression in
Theorem~\ref{thm: general measure squarefree}, which for $q=4$ and
$\mu_0(0)=\mu_1(0)=1/2$ gives
$\frac{1}{1+1/4}\left(1-\frac14\right)=\frac35.$

\begin{proof}
	Write $f=t^n+\omega g$, where $g$ is uniform among the polynomials
	in $\FF_2[t]$ of degree less than $n$, and let $m=\lfloor n/2\rfloor$.
By properties of $\FF_2$, we have unique $u,v\in\FF_2[t]$ such that $g=u^2+tv^2$. Thus 	\[
	f=
	\begin{cases}
		(t^m+\omega^2u)^2+t(\omega^2v)^2,&n=2m,\\
		(\omega^2u)^2+t(t^m+\omega^2v)^2,&n=2m+1.
	\end{cases}
	\]
	If $n=2m$, then $u$ and $v$ are independent and uniform among
	polynomials of degree less than $m$. If $n=2m+1$, then $u$ is
	uniform among polynomials of degree at most $m$, while $v$ is
	uniform among polynomials of degree less than $m$, again independently.

Let $w\in\FF_2[t]$.
The automorphism of $\FF_4[t]$ obtained by squaring every coefficient
fixes $\FF_2[t]$ and interchanges $(t^m+\omega^2w)$ with $(t^m+\omega w)$.
	Consequently, a polynomial in $\FF_2[t]$ is coprime to one of
	these factors if and only if it is coprime to their product. Now, 
$f$ is of the form $U^2+tV^2$ and has derivative $V^2$ as we work in characteristic $2$,
hence is squarefree if and only if $\gcd(U,V)=1$. So
	\begin{equation}\label{eq: F4 gcd criterion}
		f\text{ is squarefree}\quad\Longleftrightarrow\quad
		\begin{cases}
			\gcd(v,(t^m+\omega^2u)(t^m+\omega u))=1,&n=2m,\\
			\gcd(u,(t^m+\omega^2v)(t^m+\omega v))=1,&n=2m+1.
		\end{cases}
	\end{equation}
	
	Let $P\in\mathcal{P}_2$ be monic and irreducible in $\FF_2[t]$, and let $E_P$ be
	the event that $P$ divides both polynomials in the relevant gcd
	in~\eqref{eq: F4 gcd criterion}. Note that those polynomials are also contained in $\FF_2[t]$, and the $\gcd$ is 1 over $\FF_4$ iff it is 1 over $\FF_2$, so we work in $\FF_2$ here.
	
	For $P=t$ and $m\ge1$, this is
	the event $u(0)=v(0)=0$, so $\mathbb{P}(E_t)=1/4$.
	For $P\ne t$, write $d=\deg P$. In the field
	$\FF_2[t]/(P)\cong\FF_{2^d}$, the element $t$ is invertible, and
	\[
t^{2m}+t^mw+w^2
 =(t^m+\omega^2w)(t^m+\omega w)\equiv0\pmod P
	\quad\Longleftrightarrow\quad
	(w/t^m)^2+w/t^m+1=0.\]
	This equation has two roots if $d$ is even and no roots if $d$ is odd, since $\FF_{2^d}$ contains $\FF_4$ iff $2\mid d$. When $d\le m$, the residues of $u$ and $v$ modulo $P$ are independent and uniform. One must be 0 and the other must belong to a set of two residues. Thus
	\begin{equation}\label{eq: F4 prime probability}
		\mathbb{P}(E_P)=
		\begin{cases}
			2^{1-2d},&2\mid d,\\
			0,&2\nmid d,
		\end{cases}
		\qquad(P\ne t,\ d\le m).
	\end{equation}
	
Our probability estimate will be calculated as in~\eqref{eq: sec 3 start}. Take an integer $r\ge 2$.

We deal with the small primes, of degree $\le r$, first. Let $\displaystyle Q_r(t)=t\prod_{\substack{P\in\mathcal P_2^{\le r}\\P\ne t,
		\ 2\mid\deg P}}P(t).$
 If $\deg Q_r(t)\le m$, the residues of $u$ and $v$ modulo $Q_r(t)$
	are independent and uniform. The Chinese remainder theorem and \eqref{eq: F4 prime probability} then give
	\[
	\mathbb{P}\left(E_P\text{ does not occur for any }P\text{ dividing }Q_r(t)\right)
	=C_r:=\frac34\prod_{\substack{d\le r\\2\mid d}}
	(1-2^{1-2d})^{I_2(d)}.
	\]
	
	We now consider primes of degree $>r$.
	For $n=2m$, except when $v=0$, every common prime factor in \eqref{eq: F4 gcd criterion} divides the nonzero polynomial $v$, and therefore has degree at most $m-1$. For $n=2m+1$, except when $u=0$, it divides the nonzero polynomial $u$, and therefore has degree at most $m$. The probability of $v$ or $u$ being the zero polynomial is at most $2^{-m}$. For all remaining primes,~\eqref{eq: F4 prime probability} bounds contribution by $2\cdot2^{-2d}$.
We take a union bound which works for even and odd $n$:
	\begin{equation}\label{eq: F4 tail estimate}
		0\le C_r-\mathbb{P}_{f\in S_n(V)}(f\text{ is squarefree})
		\le 2^{-m}
		+2\sum_{\substack{d>r\\2\mid d}}I_2(d)\,2^{-2d}.
	\end{equation}
	Using $I_2(d)\le2^d/d$, the sum on the right is at most
	\[
	2\sum_{d>r}\frac{2^{-d}}d
	\le \frac{2^{1-r}}{r+1}.
	\]
	In particular, the product defining $\kappa_4$ converges, and
	\[
	0\le C_r-\kappa_4
	\le 2\sum_{\substack{d>r\\2\mid d}}I_2(d)\,2^{-2d}.\]
For sufficiently large $m$, choose
$r=\lfloor\log_2(m+1)\rfloor-1$. Then $\deg Q_r(t)	=1+\sum_{\substack{d\le r\\2\mid d}}dI_2(d)
\le1+\sum_{d=1}^r2^d\le m$,	 and
	$2^{-r}=O(m^{-1})$. It follows from~\eqref{eq: F4 tail estimate}
	and the preceding bounds that
	\[
	\mathbb{P}_{f\in S_n(V)}(f\text{ is squarefree})
	=\kappa_4+O(2^{-m}+2^{-r})
	=\kappa_4+O\left(\frac1n\right).
	\]
	
	Finally, $I_2(2)=1$, and
	\[
	\sum_{k\ge1}I_2(2k)\,2^{1-4k}
	\le\frac18+\sum_{k\ge2}\frac{4^{-k}}k
	\le\frac18+\frac12\sum_{k\ge2}4^{-k}
	=\frac16.
	\]
Since $\prod_j(1-a_j)\ge1-\sum_j a_j$ for $0\le a_j\le1$
, we can conclude that
	\[
	\kappa_4\ge\frac34\left(1-\frac16\right)=\frac58>\frac35.
	\]
\end{proof}

\section*{Acknowledgments}
This work was completed as part of my PhD program at the Technion -- Israel Institute of Technology. It is supported by the Israel Science Foundation (grant no. 2088/24).

I would like to thank my PhD supervisor, Dr Ofir Gorodetsky, without whose support it would not be possible to publish this paper. I have enjoyed working on this problem that he suggested and meeting with him regularly to discuss mathematics. His advice has helped me to avoid many wrong directions, and his patience and understanding have helped me through difficult times.
 
		\printbibliography[heading=bibintoc] 
	\end{document}